\documentclass[11pt]{amsart}
\usepackage{amsmath,amssymb,amsfonts}
\usepackage[T1]{fontenc}
\usepackage{needspace}
\usepackage[colorlinks=true,linkcolor=blue,citecolor=blue,urlcolor=blue]{hyperref}

\newtheorem{theorem}{Theorem}[section]
\newtheorem{lemma}[theorem]{Lemma}
\newtheorem{proposition}[theorem]{Proposition}
\newtheorem{corollary}[theorem]{Corollary}

\newcommand{\fg}{\mathfrak g}
\newcommand{\fh}{\mathfrak h}
\newcommand{\bR}{\mathbb R}
\newcommand{\bC}{\mathbb C}
\newcommand{\cH}{\mathcal H}
\newcommand{\Dmat}{\mathcal D}
\newcommand{\tr}{\operatorname{tr}}

\newcommand{\Ann}{\operatorname{Ann}}
\newcommand{\HS}{\mathrm{HS}}

\title[CHERN HSC ON TWO-STEP SOLVMANIFOLDS]{Constant Chern holomorphic sectional curvature on two-step solvmanifolds}
\author{Shanlin Guan}
\address{School of Mathematical Sciences,\newline Chongqing Normal University, Chongqing 401331, China}
\email{20269003@cqnu.edu.cn}

\author{Yuqin Guo}
\address{School of Mathematical Sciences,\newline Chongqing Normal University, Chongqing 401331, China}
\email{2024010510007@stu.cqnu.edu.cn}
\email{1942747285@qq.com}

\date{}
\subjclass[2020]{53C55}
\keywords{Chern connection, holomorphic sectional curvature, Hermitian Lie algebra, solvmanifold, two-step solvable Lie algebra}

\begin{document}
\begin{abstract}
Let $M=\Gamma\backslash G$ be a compact solvmanifold with a left-invariant Hermitian structure whose Lie algebra is two-step solvable. If the Chern holomorphic sectional curvature is constant, then the derived algebra is $J$-invariant and the Chern connection is flat. At the Lie-algebra level, we also prove that constant Chern holomorphic sectional curvature is impossible for a unimodular two-step solvable Hermitian Lie algebra of mixed Freibert--Swann type.
\end{abstract}
\maketitle

\section{Introduction}
Let $(M^n,J,g)$ be a Hermitian manifold with Chern curvature tensor $R$. The Chern holomorphic sectional curvature is
\[
H(X)=\frac{R_{X\bar XX\bar X}}{|X|^4},\qquad 0\ne X\in T^{1,0}M.
\]
For K\"ahler metrics, $H$ determines $R$. For a general Hermitian metric, it determines only the part of $R$ symmetric separately in its holomorphic and antiholomorphic indices. The classical rigidity conjecture asserted that on a compact Hermitian manifold, $H\equiv c$ forces the metric to be K\"ahler if $c\ne0$ and Chern flat if $c=0$. This is true for complex surfaces \cite{balas,balasgauduchon,satosekigawa,apostolov}. In dimension three, Ma--Nie treated normal balanced metrics in the nonpositive range, and Chen--Li removed the normality assumption \cite{manie,chenli}. Qin--Tian settled $c\ne0$ without additional assumptions, while Chen--Zheng settled $c=0$ \cite{qintian,chenzhengthree}. The zero-curvature assertion, however, fails in higher dimensions: Qin--Tian \cite{qintian} constructed compact balanced metrics with $H\equiv0$ but $R\not\equiv0$ in every complex dimension $n\ge7$. They also proved K\"ahler rigidity for $c\le0$ in Fujiki's class $\mathcal C$ and for $c\ne0$ on compact balanced fourfolds.

The conjectured rigidity also holds for locally conformally K\"ahler metrics \cite{chenchennie,huangwan} and Bismut-torsion-parallel metrics \cite{chenzhengbtp,wang,wangzhengfour,wangzhengall}. Related results for Chern--K\"ahler-like, Bismut--K\"ahler-like, and certain pluriclosed metrics appear in \cite{tangkahler,raozheng,brodertang}. Chen--Zheng also proved that for Hermitian connections with parallel torsion, nonzero constant holomorphic sectional curvature forces the torsion to vanish \cite{chenzhengparallel}. In a related direction, Tang \cite{tangreal} recently resolved the Yang--Zheng conjecture on constant real bisectional curvature \cite{yangzheng,zhouzheng}. For a survey see \cite{zheng}; related work on Chern curvature of twistor spaces and orthogonal complex structures on flat tori appears in \cite{davidovgrantcharovmuskarov,khanyangzheng}.

We consider left-invariant Hermitian metrics on compact solvmanifolds $M=\Gamma\backslash G$. Let $\fg=\operatorname{Lie}(G)$ and $\fg'=[\fg,\fg]$. The existence of the lattice implies that $\fg$ is unimodular. The role of Lie groups in flatness questions is already apparent in the theorems of Boothby and Hano \cite{boothby,hano}. Li--Zheng proved rigidity for complex nilmanifolds and studied almost-abelian and related solvable groups \cite{lizhengnil,lizhengalmost}. More directly relevant here is the theorem of Huang--Zheng \cite{huangzheng}: a compact Lie-Hermitian solvmanifold with constant Chern holomorphic sectional curvature is Chern flat when $\fg'$ is $J$-invariant. Even when $\fg'$ is abelian, this invariance need not hold. We prove that for a unimodular two-step solvable Hermitian Lie algebra it follows from the constancy of $H$ itself. Huang--Zheng's theorem then yields Chern flatness on compact quotients.

Set
\[
\fh=\fg'\cap J\fg',\qquad V=\fg'\cap\fh^\perp,\qquad W=(\fg'+J\fg')^\perp.
\]
The Freibert--Swann decomposition \cite{freibertskt,freibertbalanced} is
\begin{equation}
\label{eq:1}
\fg=\fh\oplus(V\oplus JV)\oplus W,\qquad\fg'=\fh\oplus V.
\end{equation}
These summands are mutually orthogonal, except possibly $V$ and $JV$. The cases $\fh=0$, $V=0$ and $W=0$ are called pure types I, II and III, respectively; the type is mixed if all three spaces are nonzero. Throughout, two-step solvable means that $\fg'$ is abelian, including $\fg'=0$.

Our main result is the following.

\begin{theorem}
\label{thm:main}
Let $M=\Gamma\backslash G$ be a compact Lie-Hermitian solvmanifold with two-step solvable Lie algebra. If its left-invariant Hermitian metric has constant Chern holomorphic sectional curvature, then
\[
J\fg'=\fg',\qquad H\equiv0,\qquad R=0.
\]
\end{theorem}

It remains to exclude $V\ne0$. The mixed case admits the following Lie-algebraic statement, without a lattice assumption.
\begin{theorem}
\label{thm:mixed}
Let $(\fg,J,g)$ be a unimodular two-step solvable Hermitian Lie algebra. If $\fh$, $V$ and $W$ in \eqref{eq:1} are all nonzero, then its Chern holomorphic sectional curvature cannot be constant.
\end{theorem}

The proof separates the three cases with $V\ne0$. In pure type I, unimodularity yields a nilpotent commutative associative algebra, which the curvature condition excludes in the nonabelian case. In pure type III, the corresponding quotient algebra satisfies $A^2=A$; a curvature trace identity gives a contradiction. In the mixed case, the same trace argument yields nilpotence, and further curvature identities eliminate $V$.

\section{Preliminaries}
Write $\fg^{1,0}$ for the $i$-eigenspace of $J$ in $\fg\otimes\bC$, and use a bar for complex conjugation. For matrices, $M^*=\bar M^T$ and $M^{-T}=(M^T)^{-1}$. We follow the structure-constant conventions of \cite{chenzhengstreets,huangzheng}. If $e=(e_1,\ldots,e_n)$ is a frame of $\fg^{1,0}$ with dual coframe $\varphi=(\varphi_1,\ldots,\varphi_n)$, then
\begin{equation}
\label{eq:2}
[e_i,e_j]=\sum_k C_{ij}^k e_k,\qquad
[e_i,\bar e_j]=\sum_k\bigl(\overline{D_{kj}^i}e_k-D_{ki}^j\bar e_k\bigr).
\end{equation}
For a unitary frame, the Chern curvature is \cite[Lemma 2]{huangzheng}
\begin{equation}
\label{eq:3}
R_{i\bar j k\bar\ell}=\sum_s\bigl(D_{ki}^s\overline{D_{\ell j}^s}-D_{si}^\ell\overline{D_{sj}^k}-D_{si}^j\overline{D_{\ell s}^k}-\overline{D_{sj}^i}D_{ks}^\ell\bigr).
\end{equation}
Only the part of $R$ symmetric in its two holomorphic and its two antiholomorphic entries enters $H$. Put
\begin{equation}
\label{eq:4}
\widehat R_{i\bar j k\bar\ell}=\frac14\bigl(R_{i\bar j k\bar\ell}+R_{k\bar j i\bar\ell}+R_{i\bar\ell k\bar j}+R_{k\bar\ell i\bar j}\bigr).
\end{equation}
Then $H\equiv c$ is equivalent to \cite[Lemma 1]{huangzheng}
\begin{equation}
\label{eq:5}
\widehat R_{i\bar j k\bar\ell}=\frac c2\bigl(g_{i\bar j}g_{k\bar\ell}+g_{i\bar\ell}g_{k\bar j}\bigr).
\end{equation}

\Needspace{9\baselineskip}
\begin{lemma}
\label{lem:positive}
Let $(\fg,J,g)$ be a Hermitian Lie algebra. If $0\ne x\in\fg$ satisfies $[x,Jx]=0$, then $H(x-iJx)\ge0$.
\end{lemma}
\begin{proof}
Take a unitary frame with $e_1=(x-iJx)/(\sqrt2\,|x|)$. Since $[e_1,\bar e_1]=0$, \eqref{eq:2} gives $D_{s1}^1=0$ for every $s$. Formula \eqref{eq:3} yields
\[
H(e_1)=R_{1\bar 1 1\bar 1}=\sum_s|D_{11}^s|^2\ge0.\qedhere
\]
\end{proof}

The Jacobi identity also gives the following relations \cite[(5)]{chenzhengstreets}:
\begin{equation}
\label{eq:6}
\sum_r\bigl(C_{ik}^rD_{jr}^\ell+D_{ji}^r D_{rk}^\ell-D_{jk}^r D_{ri}^\ell\bigr)=0
\end{equation}
and
\begin{equation}
\label{eq:7}
\sum_r\bigl(C_{ik}^r\overline{D_{r\ell}^j}-C_{rk}^j\overline{D_{r\ell}^i}+C_{ri}^j\overline{D_{r\ell}^k}-D_{ri}^\ell\overline{D_{jr}^k}+D_{rk}^\ell\overline{D_{jr}^i}\bigr)=0.
\end{equation}
The unimodularity condition, which holds whenever $G$ has a compact quotient, is
\begin{equation}
\label{eq:8}
\sum_s(C_{si}^s+D_{si}^s)=0\qquad\text{for every }i.
\end{equation}

For a possibly nonunitary frame, let
\[
\cH=(g_{i\bar j}),\qquad \Dmat_k=(D_{ik}^j)_{i,j}.
\]
Denote by $A_k$ and $B_k$ the matrices of $\nabla_{e_k}$ and $\nabla_{\bar e_k}$ on $\fg^{1,0}$. The connection formula of \cite[Lemma 1]{chenzhengstreets} gives
\begin{equation}
\label{eq:9}
A_k=\cH^{-T}\Dmat_k^T\cH^T,\qquad B_k=-\overline{\Dmat_k},
\end{equation}
and therefore
\begin{equation}
\label{eq:10}
R_{k\bar\ell}=A_kB_\ell-B_\ell A_k-\sum_p\overline{D_{p\ell}^k}A_p+\sum_pD_{pk}^\ell B_p.
\end{equation}
The transposes in \eqref{eq:9} matter when the frame is not unitary.

For a finite-dimensional real commutative associative algebra $A$, write $L_x(y)=x\circ y$ and set $A^1=A$, $A^{m+1}=A\circ A^m$.
\begin{lemma}
\label{lem:algebra}
Let $A$ be a finite-dimensional real commutative associative algebra.
\begin{enumerate}
\item \label{lem:algebra-i}If $\tr L_x=0$ for all $x\in A$, then $A$ is nilpotent.
\item \label{lem:algebra-ii}If $A$ is not nilpotent, there exists $y\in A$ with $L_y^2=L_y\ne0$. In particular, this holds if $A^2=A\ne0$.
\end{enumerate}
\end{lemma}
\begin{proof}
Suppose that some $L_x$ is not nilpotent. Write its minimal polynomial as $m_x(t)=t^kq(t)$ with $q(0)\ne0$. As $t^{\max\{k,1\}}$ and $q$ are coprime, there exists a real polynomial $p$ with $p\equiv0\pmod{t^{\max\{k,1\}}}$ and $p\equiv1\pmod q$. Thus $p(0)=0$ and $p(L_x)$ is a nonzero idempotent. Associativity implies $p(L_x)=L_y$ for some $y\in A$, so
\[
\tr L_y=\tr p(L_x)=\operatorname{rank}p(L_x)>0.
\]
If, instead, every $L_x$ is nilpotent, the commuting family $\{L_x:x\in A\}$ can be simultaneously strictly upper triangularized over $\bC$. Hence, with $d=\dim A$, any product of $d$ such operators vanishes and $A^{d+1}=0$. The trace hypothesis in (\ref{lem:algebra-i}) excludes the nonzero idempotent above, proving nilpotence, and nonnilpotence gives (\ref{lem:algebra-ii}). Finally, $A^2=A\ne0$ implies $A^m=A$ for every $m\ge1$.
\end{proof}

\section{Pure type I}
Suppose that $\fg'\cap J\fg'=0$, and put $A=\fg'$.
\begin{proposition}
\label{prop:pureI-algebra}
Let $\fg$ be a unimodular two-step solvable Lie algebra with complex structure $J$ and $\fg'\cap J\fg'=0$. Define
\begin{equation}
\label{eq:11}
x\circ y=[Jx,y],\qquad x,y\in A=\fg'.
\end{equation}
Then $(A,\circ)$ is a nilpotent commutative associative algebra.
\end{proposition}
\begin{proof}
The integrability of $J$ reads
\[
[X,Y]-[JX,JY]+J[JX,Y]+J[X,JY]=0.
\]
The first two terms lie in $A$, while the others lie in $JA$. Since $A\cap JA=0$,
\begin{equation}
\label{eq:12}
[JX,JY]=[X,Y],\qquad [JX,Y]=-[X,JY].
\end{equation}
In particular, $J$ is abelian and $\circ$ is commutative. Because $A$ is abelian, Jacobi gives $[L_x,L_y]z=[[Jx,Jy],z]=0$ for $x,y,z\in A$. Therefore
\[
(x\circ y)\circ z=z\circ(x\circ y)=L_zL_xy=L_xL_zy=x\circ(y\circ z).
\]
Choose a $J$-invariant complement $W$ with $\fg=A\oplus JA\oplus W$. The map $\operatorname{ad}_{Jx}$ restricts to $L_x$ on $A$, vanishes on $JA$, and maps $W$ into $A$. Thus $\tr(\operatorname{ad}_{Jx})=\tr L_x=0$ by unimodularity. Lemma~\ref{lem:algebra} proves that $A$ is nilpotent.
\end{proof}

\begin{theorem}
\label{thm:pureI}
A nonabelian unimodular two-step solvable Hermitian Lie algebra of pure type I cannot have constant Chern holomorphic sectional curvature.
\end{theorem}
\begin{proof}
Let $A=\fg'$ and use the orthogonal $J$-invariant splitting $\fg=(A\oplus JA)\oplus W$. By Proposition~\ref{prop:pureI-algebra}, the algebra $A$ is nonzero and nilpotent. If $W=0$, then \eqref{eq:12} gives $A=\fg'=[JA,A]=A^2$, a contradiction. Thus $W\ne0$. Choose a unitary frame adapted to $(A\oplus JA)^{1,0}\oplus W^{1,0}$, with $a,b,\ldots$ denoting $W^{1,0}$-indices. All brackets lie in $A_{\bC}$, so
\begin{equation}
\label{eq:13}
D_{aq}^p=0\qquad\text{for all $p,q$ and $a\in W^{1,0}$.}
\end{equation}
Also $C=0$, since $J$ is abelian. If $H\equiv c$, formula \eqref{eq:3} gives
\begin{equation}
\label{eq:14}
c=R_{a\bar a a\bar a}=-\sum_r|D_{ra}^a|^2\le0.
\end{equation}
Suppose $c=0$. Then
\begin{equation}
\label{eq:15}
D_{ra}^a=0\qquad\text{for all $r,a$.}
\end{equation}
With $C=0$, equation \eqref{eq:7} becomes
\[
\sum_r\bigl(-D_{ri}^\ell\overline{D_{jr}^k}+D_{rk}^\ell\overline{D_{jr}^i}\bigr)=0.
\]
Take $k=j$ and sum over $j$. The first term vanishes by \eqref{eq:8}, leaving
\begin{equation}
\label{eq:16}
\sum_{r,j}D_{rj}^\ell\overline{D_{jr}^i}=0.
\end{equation}
In particular,
\begin{equation}
\label{eq:17}
\sum_{r,k}D_{rk}^a\overline{D_{kr}^a}=0.
\end{equation}
Since $\widehat R=0$ by \eqref{eq:5}, expanding $4\widehat R_{a\bar a k\bar k}$ with \eqref{eq:3}, \eqref{eq:13}, and \eqref{eq:15} yields
\begin{equation}
\label{eq:18}
0=\sum_r\bigl(|D_{ka}^r|^2-|D_{ra}^k|^2-|D_{rk}^a|^2-2\operatorname{Re}(D_{rk}^a\overline{D_{kr}^a})\bigr).
\end{equation}
Summation over $k$ cancels the first two terms, and \eqref{eq:17} removes the last one. Hence $D_{rk}^a=0$ for all $a,r,k$. As $C=0$, we have $[e_a,e_j]=0$; moreover $[e_a,\bar e_j]$ has type $(0,1)$ by \eqref{eq:2}, and must vanish because it lies in $A_{\bC}$ and $A_{\bC}\cap\fg^{0,1}=0$. Thus $W$ is central, giving $A=\fg'=[JA,A]=A^2$, again a contradiction. Finally, if $c<0$, take $0\ne x$ in the last nonzero power of $A$. Then $x\circ A=0$, in particular $[Jx,x]=0$, which contradicts Lemma~\ref{lem:positive}.
\end{proof}

\section{Pure type III}
Assume now that $\fg'+J\fg'=\fg$. With $\fh=\fg'\cap J\fg'$, decomposition \eqref{eq:1} reduces to
\begin{equation}
\label{eq:19}
\fg=\fh\oplus(V\oplus JV),\qquad \fg'=\fh\oplus V.
\end{equation}
Write $r=\dim_{\bC}\fh$ and $q=\dim_{\bR}V$. Since a nonzero solvable Lie algebra cannot equal its derived algebra, $q>0$.
\begin{proposition}
\label{prop:pureIII-algebra}
The space $\fh$ is a $J$-invariant abelian ideal. The quotient $\bar\fg=\fg/\fh$ has an abelian complex structure and satisfies
\[
\bar\fg=V\oplus JV,\qquad \bar\fg'=V.
\]
Moreover, $V$ admits a commutative associative product $\circ$ with $V\circ V=V$.
\end{proposition}
\begin{proof}
The space $\fh$ is $J$-invariant and abelian by definition. For $u\in\fh$ and $v\in\fg'$, the integrability of $J$, together with $u,Ju,v\in\fg'$, gives
\[
[Jv,Ju]=J[Jv,u].
\]
Both sides lie in $\fg'$, so $[Jv,u]\in\fh$. Since $[\fg',\fh]=0$ and $\fg=\fg'+J\fg'$, we have $[\fg,\fh]\subset\fh$. Thus $J$ descends to $\bar\fg$, where \eqref{eq:19} gives $\bar\fg=V\oplus JV$ and $\bar\fg'=\fg'/\fh=V$. Here $V$ and $JV$ denote their images in $\bar\fg$. Their intersection is zero, and the argument of Proposition~\ref{prop:pureI-algebra} shows that the quotient complex structure is abelian.

Choose a real basis $x_{r+1},\ldots,x_n$ of $V$ and set
\[
e_\alpha=\tfrac12(x_\alpha-iJx_\alpha),\qquad r+1\le\alpha\le n.
\]
Together with a unitary basis $e_1,\ldots,e_r$ of $\fh^{1,0}$, these vectors form an admissible frame in the sense of \cite{chenzhengstreets}. The middle block need not be unitary. The structure relations of \cite[Lemma 4]{chenzhengstreets} include
\begin{equation}
\label{eq:20}
C_{pq}^\alpha=0,\quad \Dmat_i=0,\quad D_{\alpha\beta}^i=0,\quad C_{i\alpha}^j=-\overline{D_{j\alpha}^i},\quad D_{\gamma\alpha}^\beta=D_{\gamma\beta}^\alpha\in i\bR.
\end{equation}
Write
\begin{equation}
\label{eq:21}
D_{\gamma\alpha}^\beta=i a_{\alpha\beta}^\gamma,\qquad a_{\alpha\beta}^\gamma=a_{\beta\alpha}^\gamma\in\bR,
\end{equation}
and define
\begin{equation}
\label{eq:22}
x_\alpha\circ x_\beta=\sum_\gamma a_{\alpha\beta}^\gamma x_\gamma.
\end{equation}
The product is commutative. Equation \eqref{eq:6}, restricted to middle indices, gives $[L_\alpha,L_\beta]=0$, and hence associativity. Finally, $\bar\fg'=V$ is generated by $[JV,V]$, because $V$ and $JV$ are both abelian. As
\[
[Jx_\alpha,x_\beta]\equiv2\sum_\gamma a_{\alpha\beta}^\gamma x_\gamma\pmod\fh,
\]
we conclude that $V=V\circ V$.
\end{proof}

\begin{lemma}
\label{lem:pureIII-zero}
Let $(\fg,J,g)$ be a nonabelian unimodular two-step solvable Hermitian Lie algebra of pure type III. Then $\fh\ne0$, and in the admissible frame above,
\[
\Dmat_i=0,\qquad R_{i\bar y}=R_{x\bar i}=0\quad (i\le r)
\]
for all $x,y\in\fg^{1,0}$.
\end{lemma}
\begin{proof}
If $r=0$, there are only middle indices. Thus $C=0$, and \eqref{eq:8} and \eqref{eq:21} give
\[
0=\sum_\gamma D_{\gamma\alpha}^\gamma=i\tr L_\alpha.
\]
Hence $\tr L_x=0$ for all $x\in A$, contrary to Lemma~\ref{lem:algebra} since $A^2=A\ne0$. Thus $r>0$. For $i\le r$, \eqref{eq:20} gives $\Dmat_i=0$ and therefore $A_i=B_i=0$. In \eqref{eq:10}, all terms containing $A_i$ or $D_{pi}^y$ vanish. Since $\fh$ is an ideal, $D_{py}^i$ can be nonzero only for an $\fh$-index $p$, for which $A_p=0$. Thus $R_{i\bar y}=0$, and similarly $R_{x\bar i}=0$.
\end{proof}

\begin{theorem}
\label{thm:pureIII}
A nonabelian unimodular two-step solvable Hermitian Lie algebra of pure type III cannot have constant Chern holomorphic sectional curvature.
\end{theorem}
\begin{proof}
Suppose $H\equiv c$. By Lemma~\ref{lem:pureIII-zero}, $R_{i\bar i i\bar i}=0$ for $i\le r$, so $c=0$ and $\widehat R=0$. The same lemma and \eqref{eq:4} give
\begin{equation}
\label{eq:23}
R_{\alpha\bar\beta i\bar j}=0,\qquad 1\le i,j\le r.
\end{equation}
The metric matrix in the admissible frame is $\cH=\bigl(\begin{smallmatrix}I_r&0\\0&G\end{smallmatrix}\bigr)$ with $G=G^*>0$. For a middle index $\alpha$, write
\begin{equation}
\label{eq:24}
\Dmat_\alpha=\begin{pmatrix}P_\alpha&F_\alpha\\0&iL_\alpha\end{pmatrix},\qquad (P_\alpha)_{ij}=D_{i\alpha}^j,\quad (F_\alpha)_{i\gamma}=D_{i\alpha}^\gamma.
\end{equation}
Equation \eqref{eq:9} becomes
\begin{equation}
\label{eq:25}
A_\alpha=\begin{pmatrix}P_\alpha^T&0\\ G^{-T}F_\alpha^T&iG^{-T}L_\alpha^TG^T\end{pmatrix},\qquad
B_\beta=\begin{pmatrix}-\bar P_\beta&-\bar F_\beta\\0&iL_\beta\end{pmatrix}.
\end{equation}
Put $U_\alpha=F_\alpha^T$ and $p_\alpha=\tr P_\alpha$. Taking the trace of the $\fh^{1,0}$-block of \eqref{eq:10} and using \eqref{eq:23}, we obtain
\begin{equation}
\label{eq:26}
0=\tr(U_\beta^*G^{-T}U_\alpha)+i\sum_\gamma a_{\alpha\beta}^\gamma(p_\gamma-\bar p_\gamma).
\end{equation}
The commutator of the $P$-blocks has zero trace, and unimodularity with \eqref{eq:20}--\eqref{eq:21} gives
\begin{equation}
\label{eq:27}
(p_\gamma-\bar p_\gamma)+i\tr L_\gamma=0,\qquad i(p_\gamma-\bar p_\gamma)=\tr L_\gamma.
\end{equation}
Since $L_{x_\alpha\circ x_\beta}=L_\alpha L_\beta$, it follows that
\begin{equation}
\label{eq:28}
\tr(L_\alpha L_\beta)+\tr(U_\beta^*G^{-T}U_\alpha)=0.
\end{equation}
For $x=\sum_\alpha t_\alpha x_\alpha\in A$ with $t_\alpha\in\bR$, set $L_x=\sum_\alpha t_\alpha L_\alpha$ and $U_x=\sum_\alpha t_\alpha U_\alpha$. Contracting \eqref{eq:28} with $t_\alpha t_\beta$ yields
\begin{equation}
\label{eq:29}
\tr(L_x^2)+\tr(U_x^*G^{-T}U_x)=0.
\end{equation}
Since $G^{-T}$ is positive definite,
\begin{equation}
\label{eq:30}
\tr(L_x^2)\le0\qquad\text{for all $x\in A$.}
\end{equation}
But $A^2=A\ne0$. Lemma~\ref{lem:algebra} supplies $y\in A$ with $L_y^2=L_y\ne0$, and then $\tr(L_y^2)=\tr L_y=\operatorname{rank}L_y>0$, contradicting \eqref{eq:30}.
\end{proof}

\section{The mixed type}
To prove Theorem~\ref{thm:mixed}, let $\fg$ be unimodular and assume
\[
\fh\ne0,\qquad V\ne0,\qquad W\ne0.
\]
Put $r=\dim_{\bC}\fh^{1,0}$, $q=\dim_{\bR}V$, and $m=\dim_{\bC}W^{1,0}$. Choose an admissible frame as in \cite[Lemma 4]{chenzhengstreets}: $e_1,\ldots,e_r$ form a unitary basis of $\fh^{1,0}$, the middle vectors satisfy
\[
V=\operatorname{span}_{\bR}\{e_\alpha+\bar e_\alpha:r+1\le\alpha\le r+q\},
\]
and the remaining $e_a$ form a unitary basis of $W^{1,0}$. Throughout this section, $i,j$ denote $\fh$-indices; $\alpha,\beta,\gamma$ denote middle indices; and $a,b$ denote $W$-indices. The metric matrix has the form
\begin{equation}
\label{eq:31}
\cH=\operatorname{diag}(I_r,G,I_m),\qquad G=G^*>0.
\end{equation}
For unrestricted indices $p,q$, the structure relations in \cite[Lemma 4]{chenzhengstreets} give
\begin{equation}
\label{eq:32}
C_{pq}^\alpha=C_{pq}^a=0,\qquad D_{aq}^p=0,\qquad D_{pj}^i=D_{pj}^\alpha=0,\qquad D_{\alpha\beta}^i=0,
\end{equation}
and
\begin{equation}
\label{eq:33}
C_{j\alpha}^i=-\overline{D_{i\alpha}^j},\qquad D_{\gamma\alpha}^\beta=D_{\gamma\beta}^\alpha\in i\bR,\qquad D_{\alpha y}^x=-\overline{D_{\alpha x}^y},
\end{equation}
where $x,y$ may range over the full frame. In particular,
\begin{equation}
\label{eq:34}
\Dmat_i=\begin{pmatrix}0&0&Q_i\\0&0&Y_i\\0&0&0\end{pmatrix},\qquad
\Dmat_\alpha=\begin{pmatrix}P_\alpha&F_\alpha&Q_\alpha\\0&iL_\alpha&Y_\alpha\\0&0&0\end{pmatrix}.
\end{equation}
Here $P_\alpha,F_\alpha,Q_\alpha,Y_\alpha,L_\alpha$ have sizes $r\times r$, $r\times q$, $r\times m$, $q\times m$, and $q\times q$, respectively.

\begin{lemma}
\label{lem:mixed-zero}
If $H\equiv c$ in the mixed case, then $c=0$. Moreover, in an admissible frame,
\begin{equation}
\label{eq:35}
D_{pa}^b=0\qquad\text{for all $W$-indices $a,b$ and every $p$,}
\end{equation}
and the matrices $Z_a=(D_{ij}^a)_{i,j\le r}$ are skew-symmetric.
\end{lemma}
\begin{proof}
For a unit vector $e_i\in\fh^{1,0}$, \eqref{eq:34} and \eqref{eq:10} give
\[
R_{i\bar i i\bar i}=\sum_a|D_{ii}^a|^2\ge0.
\]
For a unit vector $e_a\in W^{1,0}$, the vanishing of the $W$-rows in \eqref{eq:32} gives $A_pe_a=0$. With $d_a=\Dmat_ae_a$, we obtain
\[
R_{a\bar a a\bar a}=-d_a^T\cH^T\overline{d_a}\le0.
\]
Both quantities equal $c$, so $c=0$ and $D_{pa}^a=0$ for every $p$. This holds for any unit vector of $W^{1,0}$. Polarizing $D_{pZ}^Z=0$ with $Z=\sum_a z_ae_a$ gives $\sum_{a,b}\bar z_az_bD_{pb}^a=0$ and hence \eqref{eq:35}. Replacing $e_i$ by an arbitrary $u\in\fh^{1,0}$ in the first identity yields $D_{uu}^a=u^TZ_au=0$. Therefore $Z_a^T=-Z_a$.
\end{proof}

Write $D_{\gamma\alpha}^\beta=i a_{\alpha\beta}^\gamma$ with $a_{\alpha\beta}^\gamma=a_{\beta\alpha}^\gamma\in\bR$. On $A=V$ define
\begin{equation}
\label{eq:36}
x_\alpha\circ x_\beta=\sum_\gamma a_{\alpha\beta}^\gamma x_\gamma,\qquad x_\alpha=e_\alpha+\bar e_\alpha.
\end{equation}
By \eqref{eq:6}, $[L_\alpha,L_\beta]=0$; as in Proposition~\ref{prop:pureI-algebra}, this makes $A$ a commutative associative algebra. In particular,
\begin{equation}
\label{eq:37}
L_{x\circ y}=L_xL_y,\qquad x,y\in A.
\end{equation}

For $x=\sum_\alpha x^\alpha x_\alpha\in A$ with real coefficients, extend the matrices in \eqref{eq:34} linearly to $P_x,F_x,Q_x,Y_x,L_x$, and put $U_x=F_x^T$ and $p_\alpha=\tr P_\alpha$. In curvature expressions, $x$ denotes the corresponding $(1,0)$-vector $\sum_\alpha x^\alpha e_\alpha$. For each $W$-index $a$, introduce the $q\times r$ matrix $K_a$ and the $r\times m$ matrix $\Xi_x$ by
\begin{equation}
\label{eq:38}
(K_a)_{\alpha i}=D_{\alpha i}^a,\qquad (\Xi_x)_{ia}=\sum_\alpha x^\alpha(GK_a)_{\alpha i}.
\end{equation}
Thus $Q_x$ and $\Xi_x$ are both $r\times m$. Unimodularity and \eqref{eq:32}--\eqref{eq:33} give
\begin{equation}
\label{eq:39}
(p_\alpha-\bar p_\alpha)+i\tr L_\alpha=0,\qquad i(p_\alpha-\bar p_\alpha)=\tr L_\alpha.
\end{equation}

\begin{lemma}
\label{lem:trace-identity}
For every real $x\in A$,
\begin{equation}
\label{eq:40}
4\sum_{i=1}^r\widehat R_{x\bar x i\bar i}=\tr(L_x^2)+\tr(U_x^*G^{-T}U_x)+\|Q_x+\Xi_x\|_{\HS}^2.
\end{equation}
Here $\|\cdot\|_{\HS}$ denotes the Hilbert--Schmidt norm.
\end{lemma}
\begin{proof}
The connection matrices in the middle direction $x$ are
\[
A_x=\begin{pmatrix}P_x^T&0&0\\G^{-T}F_x^T&iG^{-T}L_x^TG^T&0\\Q_x^T&Y_x^TG^T&0\end{pmatrix},\qquad
B_x=\begin{pmatrix}-\bar P_x&-\bar F_x&-\bar Q_x\\0&iL_x&-\bar Y_x\\0&0&0\end{pmatrix}.
\]
For $i\le r$ they are
\[
A_i=\begin{pmatrix}0&0&0\\0&0&0\\Q_i^T&Y_i^TG^T&0\end{pmatrix},\qquad
B_i=\begin{pmatrix}0&0&-\bar Q_i\\0&0&-\bar Y_i\\0&0&0\end{pmatrix}.
\]
Taking the trace over the unitary basis of $\fh^{1,0}$ in \eqref{eq:10}, the $\fh$-block of the first commutator contributes
\[
\tr\bigl(-P_x^T\bar P_x+\bar P_x P_x^T+\bar F_xG^{-T}F_x^T+\bar Q_xQ_x^T\bigr)
=\tr(U_x^*G^{-T}U_x)+\|Q_x\|_{\HS}^2.
\]
By \eqref{eq:39} and \eqref{eq:37}, the bracket-coefficient contribution is
\[
i\sum_\gamma(x\circ x)^\gamma(p_\gamma-\bar p_\gamma)=\tr L_{x\circ x}=\tr(L_x^2).
\]
The four terms in the symmetrization are therefore
\begin{align}
\sum_iR_{x\bar x i\bar i}&=\tr(L_x^2)+\tr(U_x^*G^{-T}U_x)+\|Q_x\|_{\HS}^2,\label{eq:41}\\
\sum_iR_{i\bar x x\bar i}&=\tr(Q_x^*\Xi_x),\label{eq:42}\\
\sum_iR_{x\bar i i\bar x}&=\tr(\Xi_x^*Q_x),\label{eq:43}\\
\sum_iR_{i\bar i x\bar x}&=\|\Xi_x\|_{\HS}^2.\label{eq:44}
\end{align}
Indeed, the mixed contractions are $\sum_{i,a}(Q_x)_{ia}\overline{(\Xi_x)_{ia}}$ and its conjugate, while the last one is $\sum_{i,a}|(\Xi_x)_{ia}|^2$. Adding the four identities proves \eqref{eq:40}.
\end{proof}

\begin{lemma}
\label{lem:nilpotent}
The algebra $A=(V,\circ)$ is nilpotent. Moreover, for every real $x\in A$,
\begin{equation}
\label{eq:45}
F_x=0,\qquad Q_x+\Xi_x=0.
\end{equation}
\end{lemma}
\begin{proof}
Lemma~\ref{lem:mixed-zero} gives $c=0$, hence $\widehat R=0$. Since $G^{-T}$ is positive definite, \eqref{eq:40} implies $\tr(L_x^2)\le0$ for all $x\in A$. If $A$ were not nilpotent, Lemma~\ref{lem:algebra} would give $y\in A$ with $L_y^2=L_y\ne0$ and hence $\tr(L_y^2)=\operatorname{rank}L_y>0$, a contradiction. Thus $A$ is nilpotent and $\tr(L_x^2)=0$ for every $x$. The other two terms in \eqref{eq:40} are nonnegative, so both vanish, proving \eqref{eq:45}.
\end{proof}

Let $\Ann(A)=\{z\in A:z\circ A=0\}$. For every $W$-index $a$, put
\begin{equation}
\label{eq:46}
S_a=(D_{\alpha\beta}^a)_{\alpha,\beta},\qquad Y_x=(S_1x,\ldots,S_mx).
\end{equation}
We use $\|v\|_{G^{-1}}^2=v^*G^{-1}v$ for $v\in\bC^q$ and $\|v\|_G^2=v^TGv=v^*Gv$ for real $v$. In particular, $\|v\|_G>0$ for every nonzero real $v$.

\begin{lemma}
\label{lem:annihilator}
For $u\in A^2$ and $z\in\Ann(A)$,
\begin{equation}
\label{eq:47}
u^TGz=0.
\end{equation}
Consequently $A^2=0$.
\end{lemma}
\begin{proof}
By \eqref{eq:45}, the $\fh$-block decouples when all four curvature arguments are middle vectors. The middle block of \eqref{eq:10} is
\[
-G^{-T}L_x^TG^TL_x+L_xG^{-T}L_x^TG^T+\bar Y_xY_x^TG^T-G^{-T}(L_x^2)^TG^T-L_x^2.
\]
Contracting with the real middle vector $x$ gives
\begin{equation}
\label{eq:48}
\begin{aligned}
R_{x\bar x x\bar x}&=\|L_x^TGx\|_{G^{-1}}^2+\sum_a|x^TGS_ax|^2\\
&\quad-\|x\circ x\|_G^2-2\operatorname{Re}(x^TGL_x^2x).
\end{aligned}
\end{equation}
For $z\in\Ann(A)$, $L_z=0$, and \eqref{eq:48} with $x=z$ gives
\begin{equation}
\label{eq:49}
z^TGS_az=0\qquad\text{for every $a$.}
\end{equation}
Fix $y\in A$ and replace $x$ by $y+tz$, $t\in\bR$. Then $L_x=L_y$, $x\circ x=y\circ y$, and $L_y^2z=0$. By \eqref{eq:49}, the coefficient of $t^2$ in \eqref{eq:48} is
\[
\|L_y^TGz\|_{G^{-1}}^2+\sum_a|z^TGS_ay+y^TGS_az|^2.
\]
Since $R_{x\bar x x\bar x}=0$, this coefficient vanishes. Positive definiteness gives $L_y^TGz=0$ for every $y$, and therefore $(y\circ u)^TGz=0$ for all $y,u\in A$. This proves \eqref{eq:47}. If $A^2\ne0$ and $N\ge3$ is the nilpotency index, then $0\ne A^{N-1}\subset A^2\cap\Ann(A)$. Taking $0\ne z\in A^{N-1}$ and $u=z$ in \eqref{eq:47} gives $z^TGz=0$, a contradiction.
\end{proof}

\Needspace{7\baselineskip}
Fix a $W$-index $a$ and set
\begin{equation}
\label{eq:50}
E_a=(D_{pk}^a)_{p,k\le r+q}=\begin{pmatrix}Z_a&Q^{(a)}\\K_a&S_a\end{pmatrix},\qquad
\cH_0=\operatorname{diag}(I_r,G),
\end{equation}
where $Z_a=(D_{ij}^a)$ and $Q^{(a)}=(D_{i\alpha}^a)$.

\Needspace{9\baselineskip}
Since $A^2=0$, all of $A$ belongs to $\Ann(A)$, and \eqref{eq:48} implies $x^TGS_ax=0$ for every real $x\in A$. Thus
\begin{equation}
\label{eq:51}
(GS_a)^T=-GS_a.
\end{equation}
Also $Z_a^T=-Z_a$ by Lemma~\ref{lem:mixed-zero}, while \eqref{eq:45} and \eqref{eq:38} give $Q^{(a)}=-(GK_a)^T$. Consequently,
\begin{equation}
\label{eq:52}
(\cH_0E_a)^T=-\cH_0E_a.
\end{equation}

\Needspace{9\baselineskip}
\begin{lemma}
\label{lem:E-zero}
Suppose $H\equiv0$, $D_{pa}^b=0$ for all $W$-indices $a,b$ and every $p$, and $(\cH_0E_a)^T=-\cH_0E_a$. Then $E_a=0$ for every $a$.
\end{lemma}
\begin{proof}
We can make the non-$W$ block unitary without changing the $W$-frame, since the curvature contraction below is tensorial. Choose $T$ satisfying
\[
T\in\operatorname{GL}(r+q,\bC),\qquad T^T\cH_0\bar T=I,
\]
and put $e_p'=\sum_kT_{kp}e_k$ for $p\le r+q$. The matrices with upper index $a$ transform as
\[
E_a'=\bar T^{-1}E_aT=T^T\cH_0E_aT,
\]
which is skew-symmetric by \eqref{eq:52}. This change of frame need not preserve the real middle space, but it preserves the curvature trace and the vanishing of the $W$-rows.

In the resulting unitary frame, sum the four terms of $4\widehat R_{a\bar a k\bar k}=0$ over $k\le r+q$. By \eqref{eq:3},
\[
0=\sum_{k,p\le r+q}\bigl(|D_{ka}^{\prime p}|^2-|D_{pa}^{\prime k}|^2\bigr)-\|E_a'\|_{\HS}^2-2\operatorname{Re}\tr(E_a'\overline{E_a'}).
\]
The double sums cancel. Skew-symmetry gives
\[
\tr(E_a'\overline{E_a'})=\sum_{p,k}E_{a,pk}'\overline{E_{a,kp}'}=-\|E_a'\|_{\HS}^2.
\]
It follows that $\|E_a'\|_{\HS}^2=0$, and hence $E_a=0$.
\end{proof}

\begin{proof}[Proof of Theorem~\ref{thm:mixed}]
By Lemma~\ref{lem:nilpotent}, $A$ is nilpotent. Lemma~\ref{lem:annihilator} gives $A^2=0$, and Lemma~\ref{lem:E-zero} then implies $E_a=0$ for all $a$. Equations \eqref{eq:34}, \eqref{eq:45}, and \eqref{eq:50} reduce the matrices to
\[
\Dmat_i=0,\qquad \Dmat_\alpha=\begin{pmatrix}P_\alpha&0&0\\0&0&0\\0&0&0\end{pmatrix},\qquad \Dmat_a=\begin{pmatrix}*&*&0\\0&0&0\\0&0&0\end{pmatrix}.
\]
Thus $D_{\alpha k}^j=0$ for every middle first lower index $\alpha$ and all $j,k$. Also $C_{jk}^\alpha=0$ by \eqref{eq:32}. The bracket formula \eqref{eq:2} therefore shows that no Lie bracket has a real $V$-component, contrary to $\fg'=\fh\oplus V$ with $V\ne0$.
\end{proof}

\Needspace{16\baselineskip}
\begin{proof}[Proof of Theorem~\ref{thm:main}]
The abelian case is Chern flat. Otherwise, if $\fh=0$ we apply Theorem~\ref{thm:pureI}; if $W=0$, Theorem~\ref{thm:pureIII}; and if $\fh$, $V$, and $W$ are all nonzero, Theorem~\ref{thm:mixed}, since a compact quotient implies unimodularity. Hence $V=0$, or $J\fg'=\fg'$. The result of Huang--Zheng \cite{huangzheng} gives $H\equiv0$ and $R=0$.
\end{proof}
The nonexistence arguments use the lattice only through unimodularity, and therefore also give the following Lie-algebraic conclusion.

\begin{corollary}
Let $(\mathfrak g, J, g)$ be a unimodular two-step solvable Hermitian Lie algebra with constant Chern holomorphic sectional curvature. Then $J\mathfrak g' = \mathfrak g'$.
\end{corollary}

Theorem \ref{thm:main} concerns the original invariant metric. It makes no assertion about non-invariant metrics on the underlying complex solvmanifold, since averaging does not preserve constant holomorphic sectional curvature.

\par\vspace{0.72\baselineskip}
\Needspace{7\baselineskip}
\noindent\textbf{Acknowledgments.}\enspace We thank Professor Fangyang Zheng for his guidance and helpful discussions.\par\smallskip

\noindent\textbf{Generative AI disclosure.}\enspace
The authors used OpenAI ChatGPT as an assisting tool in the development of the paper. The authors take full responsibility for the contents of the paper.\par\smallskip

\noindent\textbf{Declaration on competing interests.}\enspace
All authors declare that there are no competing interests for this article.

\end{document}